\documentclass{sample}
\usepackage[dvips]{color}
\usepackage{color}

\title{Reducts of the Generalized Random Bipartite Graph}
\author{Yun Lu}

\address{Department of Mathematics\\
Kutztown University of PA\\
Kutztown, PA 19530}
\email{lu@kutztown.edu}
\thanks{This paper is a revision of the author's doctoral dissertation, written under the direction of Carol Wood at Wesleyan University. The author
would like to thank Dr.Wood for her guidance and for helpful comments on earlier versions of this paper.}

\primarydata{03C99}
\secondarydata{05C99}


\newtheorem{thm}{Theorem}[section]

\newtheorem{lem}[thm]{Lemma}
\newtheorem{prop}[thm]{Proposition}

\newtheorem{defn}[thm]{Definition}

\begin{document}

\begin{abstract}
Let $\Gamma$ be the generalized random bipartite graph that has two sides $R_l$ and $R_r$ with edges for every pair of vertices between $R_l$ and $R_r$ but no edges within each side, where all the edges are randomly colored by three colors $P_1, P_2, P_3$. In this paper, we investigate the reducts of $\Gamma$ that preserve $R_l$ and $R_r$, and classify the closed permutation subgroups in $Sym(R_l)\times Sym(R_r)$ containing the group $Aut(\Gamma)$.  Our results rely on a combinatorial theorem of Ne\v{s}et\v{r}il-R\"{o}dl and the strong finite submodel property of the random bipartite graph.
\end{abstract}
\maketitle

\section{Introduction}

 As in \cite{sT96}, a reduct of a structure $\Gamma$ is a structure with the same underlying set as $\Gamma$ in some relational language, each of whose relation is $\emptyset$-definable in the original structure. If $\Gamma$ is $\omega$-categorical, then a reduct of $\Gamma$ corresponds to a closed permutation subgroup in $Sym(\Gamma)$ (the full symmetric group  on the underlying set of $\Gamma$) that contains  $Aut(\Gamma)$ (the automorphism group  of $\Gamma$). Two interdefinable reducts are considered to be equivalent. That is, two reducts of a structure $\Gamma$ are equivalent if they have the same $\emptyset$-definable sets, or, equivalently, they have the same automorphism groups. There is a one-to-one correspondence between equivalence classes of reducts $N$ and closed subgroups of $Sym(\Gamma)$ containing $Aut(\Gamma)$ via $N\mapsto Aut(N)$ (see \citend{sT96}).

There are currently a few $\omega$-categorical structures whose reducts have been explicitly classified. In 1977, Higman classified the reducts of the structure $(\mathbb{Q}, <)$ (see \citend{gH77}). In 2008, Markus Junker and Martin Ziegler classified the reducts of expansions of $(\mathbb{Q}, <)$ by constants and unary predicates (see \citend{mJ05}). Simon Thomas showed that there are finitely many reducts of the random graph (\citend{sT91}) in 1991, and of the random hypergraphs (\citend{sT96}) in 1996. In 1995 James Bennett proved similar results for the random tournament, and for the random $k$-edge coloring graphs (\citend{jB95}). In 2011, I investigated the reducts of the random bipartite graph that preserve sides. Equivalently, we analyze the closed subgroups of $Sym(R_l)\times Sym(R_r)$ containing $Aut(\Gamma)$.

In this paper, we consider the generalized random bipartite graphs, i.e. complete random bipartite graphs with $k$ colors $P_1$, $P_2, P_3$ on
$R_l\times R_r$ such that $P_1\cup P_2\cup P_3=R_l\times R_r$ and $P_i\cap P_j=\emptyset$ if $i\neq j$. The appropriate language $L_3$  for such structures can be taken to have two unary relations, $R_l$ and $R_r$, and $3$ binary relations $P_1, P_2, P_3$. For convenience we consider a graph $\Gamma =(V, R_l, R_r, P_1, P_2, P_3)$, where $R_l, R_r\subseteq V$ and $P_1, P_2, P_3\subseteq R_l\times R_r$. Then $\Gamma$ is a bipartite graph having $3$ cross-types if it satisfies the following set $B_3$ of axioms:

\begin{enumerate}
\item $\exists x R_l(x)$
\item $\exists x R_r(x)$
\item $\forall x\forall y (P_i(x, y)\longrightarrow (R_l(x)\wedge R_r(y))), i = 1,\dots,3$
 \item $\forall x\forall y (P_i(x, y)\longrightarrow\neg P_j(x, y)),  i \neq j$

\item $\forall x\forall y(R_l(x)\wedge R_r(y)\longrightarrow(P_1(x, y)\vee P_2(x, y)\vee P_3(x, y)))$;
\item $\forall x ((R_l(x)\vee R_r(x))\wedge\neg (R_l(x)\wedge R_r(x)))$.
\end{enumerate}

\begin{defn}{\label{gbipartite}}
A countable bipartite graph $\Gamma$ having $3$ cross-types $\Gamma$ is \textit{random} if it satisfies the extension properties $\Theta_n$ for all  $n\in \mathbb{N}$:

($\Theta_n$): For any finite pairwise disjoint $X_{l1}$, $X_{l2}, X_{13}\subset R_l$ and finite
pairwise disjoint $X_{r1}$, $X_{r2}, X_{r3}\subset R_r$, each of size at most $n$,
\renewcommand{\labelenumi}{(\alph{enumi})}
\begin{enumerate}
\item there exists a vertex $v\in R_l$ such that $P_i(v, x)$ for every
$x\in X_{ri}, i = 1,\dots,3.$
\item there exists a vertex $w\in R_r$ such that $P_i(x, w)$ for every $x\in X_{li}, i= 1,\dots, 3.$
\end{enumerate}
\end{defn}

The $\Theta_n$'s are first-order sentences, and the axioms in Definition \ref{gbipartite} together with the  $\{\Theta_n\}_{n\in\mathbb{N}}$ form a complete and $\omega$-categorical theory. It can be shown that a 3-colored random bipartite graph exist. It is countable and unique up to isomorphism. It is also easy to show that the 3-colored random bipartite graph is homogeneous by a back-and-forth argument. In the rest of paper, the we use $\Gamma$ to denote the 3-colored random bipartite graph, unless otherwise mentioned.
Notice that with three cross-types, the definition of switch is more complicated because the permutation group $S_3$ is not commutative. From now on, we let $Sym_{\{l, r\}}(\Gamma)$ denote $Sym(R_l)\times Sym(R_r)$.
\begin{defn}
Given $\sigma\in S_3$ and a vertex $v\in R_l$, a \textit{switch} on $v$
according to $\sigma$ is a permutation $\pi\in Sym_{\{l, r\}}(\Gamma)$ such that for any $(a, b)\in R_l\times R_r$ and for $i=1, 2,
3$,
\begin{itemize}
\item if $v= a$, then $P_i(a, b)\longrightarrow
P_{\sigma(i)}(\pi(a), \pi(b))$;
\item otherwise, $P_i(a, b)\longrightarrow
P_{i}(\pi(a), \pi(b))$.
\end{itemize}
\end{defn}
Similarly we define a switch  w.r.t. $v\in R_r$.
\begin{defn}
Given $\sigma\in S_3$ and  $A\subset
\Gamma$, a \textit{switch} on $A$  according to
$\sigma$ is a permutation $\pi\in Sym_{\{l, r\}}(\Gamma)$ such that
for any $(a, b)\in R_l\times R_r$ and for $i=1, 2$ or $3$,
\begin{itemize}
\item if $(a, b)$ has
exactly one entry from  $A$, then $P_i(a, b)\longrightarrow
P_{\sigma(i)}(\pi(a), \pi(b))$;
\item if $(a, b)$ has both entries in $A$, then
$P_i(a, b)\longrightarrow P_{\sigma^2(i)}(\pi(a), \pi(b))$;
\item otherwise,
$P_i(a, b)\longrightarrow P_i(\pi(a), \pi(b))$.
\end{itemize}
\end{defn}

\begin{defn}
If $X\subseteq{\{l, r\}}$ and $H \leq S_3$, then \textit{$S^H_X(\Gamma)$} is the closed subgroup
of $Sym_{\{l, r\}}(\Gamma)$ generated as a topological group by
  Aut($\Gamma$) together with  all $\pi\in Sym_{\{l, r\}}(\Gamma)$ such that there exists a
vertex $v\in R_i$ for $i\in X$ and $\sigma \in H$ such that $\pi$ is a switch w.r.t
$v$ according to $\sigma$.

\end{defn}

Thus the candidates for the reducts are $S_{\{l\}}^H(\Gamma)$,
$S_{\{r\}}^H(\Gamma)$ and $S_{\{l, r\}}^H(\Gamma)$, where $H$ is one of the subgroups of the permutation group $S_3$:

$\{(1)\}$, $\{(1), (12)\}$, $\{(1), (13)\}$, $\{(1),
(23)\}$, $\{(1), (123), (132)\}$ and $S_3$.

We begin the analysis of reducts of $\Gamma$, the random bipartite graph  with three cross-types, by indicating which reducts are essentially new; these will be the ones we call irreducible.

\begin{defn}\label{irre}
Let $G$ be a closed subgroup of $\Gamma$ in $Sym (\Gamma)$.

We say $G$ is \textit{reducible} if
for some $k \in \{1,2,3\}$,  G contains every map $g\in Sym (\Gamma)$

 which  preserves $P_k$.
This means that G is blind to the distinction between  the other two cross-types.
If $G$ is not reducible, then we say $G$ is irreducible.
\end{defn}
So if $G$ is  reducible, then $G$ can be viewed as a reduct of the bipartite graph with two edges,
as already classified in the previous chapter.

Here is the main result of this paper:

\begin{thm}\label{final}
If $G$ is an irreducible closed subgroup such that $Aut(\Gamma)\leq G\leq
Sym_{\{l, r\}}(\Gamma)$, then $G=\langle S_{\{l\}}^{H_1}(\Gamma),
S_{\{r\}}^{H_2}(\Gamma)\rangle$ where $H_1, H_2\leq S_3$. If  $G\subset
Sym_{\{l, r\}}(\Gamma)$, then $H_1=H_2$ unless one of the two groups is trivial.
\end{thm}

Here is how the rest of the paper is organized. ???In section 2, we show that the 3-colored random bipartite graph has the strong finite submodel property; In section 3, we study the relations preserved by the groups $S_X{(\Gamma)}$, where $X\subseteq\{l, r\}$. In section 3,  and in section 4, we discuss a technique term $(m\times n)$-analysis and prove its existence for the random bipartite graph. These prepare us to give an explicit classification of the closed subgroups of $Sym_{\{, r\}}(\Gamma)$ containing $Aut(\Gamma)^*$ in the rest of the paper. In section 5, we prove the first part of Theorem \ref{classification}, which says that the closed subgroups of $S_{\{l, r\}}(\Gamma)$ containing $Aut(\Gamma)^*$ are $Aut(\Gamma)^*, S_{\{l\}}(\Gamma)$, and $S_{\{r\}}(\Gamma)$, and $S_{\{l, r\}}(\Gamma)$. Then in section 6 we show there is no other proper closed subgroup between $S_{\{l, r\}}(\Gamma)$ and $Sym_{\{l, r\}}(\Gamma)$, which completes the proof of Theorem \ref{classification}.

\section{Strong Finite Submodel Property}
In this section, we use the notion of the Strong Finite Submodel Property (SFSP) initially introduced by Thomas in \citend{sT96}, and we prove that the random 3-colored bipartite graph has the SFSP. This property provides a powerful tool when it comes to the proof in the later sessions.
\begin{defn}[\citend{sT96}]
A countable infinite structure $\mathcal{M}$ has the
\textit{Strong Finite Submodel Property (SFSP)} if $\mathcal{M}=\bigcup_{i\in\mathbb{N}} \mathcal{M}_i$ is a union of an increasing chain of substructures $\mathcal{M}_i$ such that
\renewcommand{\labelenumi}{(\arabic{enumi})}
\begin{enumerate}
\item $|\mathcal{M}_i|=i$ for each $i\in\mathbb{N}$; and
\item for any sentence $\phi$ with $\mathcal{M}\models \phi$, there exists $N\in \mathbb{N}$ such that $\mathcal{M}_i\models \phi$ for all $i\geq N$.
\end{enumerate}
\end{defn}

Here we choose a specific chain of bipartite graphs $\Gamma_i$ such that $\Gamma=\cup \Gamma_i$ where $|\Gamma_i|=i$,
$\Gamma_i\subset \Gamma_{i+1}$ for $i\in \mathbb{N}$, and
\begin{itemize}
\item if $i$ is even, then $|\Gamma_{i}\cap R_l|=|\Gamma_{i}\cap R_r|$;
\item otherwise, $|\Gamma_{i}\cap R_l|=|\Gamma_{i}\cap R_r|+1$.
\end{itemize}

Thus for any sentence $\phi$ true in $\Gamma$, there is an $j_{\phi}$ such that $i>j_{\phi}$ implies $\phi$ is true in $\Gamma_i$.

\begin{thm}\label{sfsp}
The countable random 3-colored bipartite graph $\Gamma$ has the SFSP.
\end{thm}

Theorem \ref{sfsp} is a consequence of the Borel--Cantelli Lemma, as follows below:

\begin{defn}[\citend{sT96}]
If $\{A_n\}_{n\in \mathbb{N}}$ is a sequence of events in
a probability space, then $\bigcap_{n\in
\mathbb{N}}[\bigcup_{n\leq k\in \mathbb{N}}A_k]$ is the event that consists of realization of infinitely many of $A_n$, denoted by \textit{$\overline\lim A_n$}.
\end{defn}

\begin{lem}[Borel--Cantelli, \citend{pB79}]\label{bcl}
Let $\{A_n\}_{n\in \mathbb{N}}$ be a sequence of events
in a probability space. If $\sum_{n=0}^{\infty}{P(A_n)}<\infty$,
then $P(\overline\lim A_n)=0$.
\end{lem}

\begin{proof}[Proof of Theorem \ref{sfsp}]
 Since the extension properties $\Theta_n$'s axiomatize the random 3-colored bipartite graph $\Gamma$ and $\Theta_i$ implies $\Theta_{i-1}$ for all $i\in \mathbb{N}$, for every sentence $\phi$ true in $\Gamma$, there exists some $k\in \mathbb{N}$ such that $\Theta_k$ holds if and only if $\phi$ holds. Let $\Omega$ be the probability space of all countable bipartite graphs $(S, R_l, R_r, P_1, P_2, P_3)$, where $| R_l|=| R_r|=\omega$ and every cross-edge $E\in R_l\times R_r$ has the cross-type $P_1, P_2$ or $P_3$ on it independently with probability $\frac{1}{3}$. For each $n\in \mathbb{N}$ with $n\geq k$,  let $A_n$ be the event that a 3-colored bipartite subgraph $S_n\in [S]^n$ does not satisfy the extension property $\Theta_k$. We consider two cases: n is even (n=2m), and n is odd (n=2m+1). Then by simple computation,
 \begin{equation}
P(A_{2m})\leq 2*{{m}\choose{k}}{{m-k}\choose{k}}{{m-2k}\choose{k}}({(1-(\frac{1}{3})^{3k})})^{m-3k}\nonumber,
\end{equation}
and
 \begin{equation}
P(A_{2m+1})\leq 2*{{m+1}\choose{k}}{{m+1-k}\choose{k}}{{m+1-2k}\choose{k}}({(1-(\frac{1}{3})^{3k})})^{m-3k}\nonumber.
\end{equation}
Notice that $\sum_{n=0}^{\infty}{P(A_n)}=\sum_{m=0}^{\infty}{P(A_{2m})}+\sum_{m=0}^{\infty}{P(A_{2m+1})}$, we have
\begin{equation}
\sum_{n=0}^{\infty}{P(A_n)}\leq 4*\sum_{m=0}^{\infty}{{{m+1}\choose{k}}{{m+1-k}\choose{k}}{{m+1-2k}\choose{k}}({(1-(\frac{1}{3})^{3k})})^{m-3k}}
\end{equation}
where ${{n}\choose{k}}$ is the number of combinations of n objects taken $k$ at a time. Let $C_m={{{m+1}\choose{k}}{{m+1-k}\choose{k}}{{m+1-2k}\choose{k}}({(1-(\frac{1}{3})^{3k})})^{m-3k}}$. Then $\lim_{m\to +\infty}\frac{C_{m+1}}{C_m}=(1-(\frac{1}{3})^{3k})<1$. By the ratio test for infinite series, we have $\sum_{m=0}^{\infty}{C(m)}$ converges, and so does $\sum_{n=0}^{\infty}{P(A_n)}$. Thus by Lemma \ref{bcl}, $P(\overline\lim A_n)=0$. So there exists a 3-colored bipartite graph $S\in\Omega$ and an integer $N$ such that for all $n\geq N$, $S_n\in [S]^n$ satisfies the extension property $\Theta_k$, and so $\phi$. Notice that the choice of $S$ ensures that $S$ is countable and satisfies all the axioms for the random bipartite graph. Hence $S$ is isomorphic to $\Gamma$. Then $\Gamma$ has the SFSP, which completes the proof of Theorem \ref{sfsp}.
\end{proof}

Similarly, we can show as in Theroem \ref{sfsp}  that
\begin{prop}
The countable random $k$-colored bipartite graph $\Gamma$ has the strong finite submodel property (SFSP).
\end{prop}

For the remainder of this Chapter we will restrict our attention to the case $k = 3$.
We expect that our results to generalize to arbitrary $k$, but we have not
organized the details for the more general results at this stage.

\section{Candidates for Irreducible Closed Groups}
In this section, we will discuss the candidates for irreducible closed groups.

Motivated by the colorings defined in Bennet's thesis \cite{jB95}, we may define a class of edge colorings $\chi: [A]^2\longrightarrow \{l, r, P_1, P_2, P_3\}$ for a bipartite graph $A\subseteq\Gamma$ as follows,  where $\{a, b\}\in [A]^2$:
\begin{itemize}
\item if $\{a, b\}\in [R_l]^2$, then $\chi(a, b)=l$;
\item if $\{a, b\}\in [R_r]^2$, then $\chi(a, b)=r$;
\item if $(a, b)\in R_l\times R_r$ and $P_i(a, b)$ for some $i=1, 2, 3$, then $\chi(a, b)=P_i$.
\end{itemize}

\begin{defn}
 Let $A_1$ be a bipartite graph with the edge coloring $\chi_1$, and $A_2$ be a bipartite graph with the edge coloring $\chi_2$ where $\chi$'s are defined as above. If $|A_1|=|A_2|$ and $|A_1\cap R_l|=| A_2\cap R_l|$, then the edge coloring $\chi_2$ is a \textit{permutation}
 of the edge coloring $\chi_1$ if there is some vertex bijection $\phi:
 A_1\longrightarrow A_2$ preserving $R_l, R_r$ and some permutation $\sigma\in S_3$
 such that for every cross-edge $(a, b)\in (A_1\times A_1) \cap (R_l\times R_r)$, $\chi_1(a, b)=\sigma(\chi_2(\phi(a),  \phi(b))$. That is, $P_i(a, b)$ implies $P_{\sigma(i)}(\phi(a), \phi(b))$ for $i=1, 2, 3$.

\end{defn}

\begin{defn}

Let $A$ be  a bipartite graph, and $\chi_1, \chi_2$  be  edge colorings on $[A]^2$   as above. Then the edge coloring $\chi_2$ is \textit{homogeneous} w.r.t. the coloring $\chi_1$ if for any $(a, b), (a', b')\in  (A\times A) \cap (R_l\times R_r)$, $\chi_2(a, b)=\chi_2(a', b')$ implies $\chi_1(a, b)=\chi_1(a', b')$.
\end{defn}

\begin{claim}\label{homogeneous}
If $A$ is a bipartite graph, $\chi_1$ and $\chi_2$ are edge colorings on $[A]^2$ defined as above. If $\chi_2$ is homogenous w.r.t. $\chi_1$ but is not a permutation of $\chi_1$, then there must be two distinct colors $P_i$ and $P_j$ and some color $P_k$ ($i, j, k\in \{1, 2, 3\}$) such that for any $(x, y)\subseteq (A\cap R_l)\times (A\cap R_r)$, $\chi_2(x, y)=P_i$ or $\chi_2(x, y)=P_j$ implies $\chi_1(x, y)=P_k$.
\end{claim}

\begin{proof}
By the definition of homogeneous and permutation colorings, if $\chi$ is not a permutation then
it must collapse two colors.
\end{proof}

\begin{defn}
Let $G$ be an irreducible closed subgroup of $Sym (\Gamma)$. The pair $(R, \alpha)$, where $R$ is a finite subset of $\Gamma$ and  $\alpha\in \Gamma\backslash R$, is
\textit{sufficiently complex} w.r.t. $G$ if the following hold:
\renewcommand{\labelenumi}{(\arabic{enumi})}
\begin{enumerate}
\item for any $g\in G$ and $c\in\{P_1, P_2, P_3\}$, there is a
cross-edge $(a, b)\in (g(R)\cap R_l)\times (g(R)\cap R_r)$ such that
$c(a, b)$. ($R$ witnesses all the cross-types.)
\item if there is some $g\in G$ such that $g\upharpoonright {R\cup\{\alpha\}}$
is a switch w.r.t. $\alpha\in R_l$ permuting $P_i$ according to
$\sigma$, then some (hence all) switches $f$ w.r.t. singletons in $R_l$ according to $\sigma$ are also in $G$. ($R\cup\{\alpha\}$ witnesses
which switches w.r.t. $R_l$ are not in $G$.)
\item if there is some $g\in G$ such that $g\upharpoonright {R\cup\{\alpha\}}$
is a switch w.r.t. $\alpha\in R_r$ permuting $P_i$ according to
$\sigma$, then some (hence all) switches $f$ w.r.t. singletons in $R_r$ according to $\sigma$ are also in $G$. ($R\cup\{\alpha\}$ witness
which switches w.r.t. $R_r$ are not in $G$.)
\end{enumerate}
\end{defn}
Note: This ``sufficiently complex" concept is different from that in Chapter $3$.

\begin{defn}
A finite $R\subset\Gamma$ is \textit{sufficiently complex} w.r.t. $G$ if it satisfies Property $(1)$, and there is some $\alpha\in R_l$ such that $(R, \alpha)$ satisfies Property $(2)$, and there is some $\beta\in R_r$ such that $(R, \beta)$ satisfies Property $(3)$.
\end{defn}

\begin{claim}\label{S}
If $R$ is sufficiently complex and $S\supseteq R$, then S is sufficiently complex.
\end{claim}

\begin{proof}

First, for any $g\in G$, and any $c\in \{P_i\}$ for $i=1, 2, 3$, since $R$ is
sufficiently complex, there always exists one cross-edge
$(a, b)\in g(R)\subseteq g(S)$ such that $c(a, b)$. So $S$ has the property
$(1)$. Second, suppose $\alpha\in \Gamma\backslash R$ and $(R, \alpha)$ is sufficiently complex. If there exists some $g\in G$ such that
$g\upharpoonright {S\cup\{\alpha\}}$ is a switch w.r.t. $\alpha\in R_l$
according to $\sigma$, then $g\upharpoonright {R\cup\{\alpha\}}$ is a switch
w.r.t. $\alpha$ according to $\sigma$. Since $(R,\alpha)$ is sufficiently
complex, some (hence all) switches $f$ w.r.t. singletons according to $\sigma$ are in $G$, showing that condition $(2)$ holds.  Similarly we can prove $(S, \alpha)$ has
the property $(3)$.
\end{proof}

\begin{thm}\label{existstence}
If $G$ is an irreducible closed subgroup of $Sym_{\{l, r\}}(\Gamma)$ containing $Aut(\Gamma)$, then there is a pair $(R, \alpha)$ which
is sufficiently complex w.r.t. $G$.
\end{thm}

\begin{proof}
We prove that there exist $R_0$ with the property $(1)$, $(R_1, \alpha)$ where $\alpha\in R_l$ with the property $(2)$, and $(R_2, \alpha)$ where $\alpha\in R_r$ with the property $(3)$.

\textbf{Property (1)}. Suppose there is no such $R_0$ in $\Gamma$: i.e. for any
$R\subset \Gamma$, there exists some $g\in G$ such that the
cross-edges of $g(R)$ have fewer than three cross-types. Then
if $\Gamma=\cup \Gamma_i$ is our nice enumeration as in the strong finite submodel property, there exists
a sequence $\{f_i\}_{i\in \mathbb{N}}\subset G$ such that $f_i(\Gamma_i)$ has fewer than three
cross-types. Since there are only three cross-types, but
infinitely many $\{f_i\}$, then there exists some $c\in\{P_i\}$ and
$\{f_{i_j}\}\subseteq \{f_{i}\}$ such that $f_{i_j}(\Gamma_{i_j})$ has no cross-type
$c$. Hence for every finite $B\subset\Gamma$, there is some $g\in G$
such that $g(B)$ has no cross-edge with cross-type $c$ on it. For each $i\in \mathbb{N}$, we define an edge coloring $\chi: [\Gamma_i]^2\longrightarrow \{l, r, P_1, P_2, P_3\}$ for every $\{a, b\}\in [\Gamma_i]^2$ by
\begin{itemize}
\item if $\{a, b\}\subseteq R_l$, then $\chi(a, b)=l$;
\item if $\{a, b\}\subseteq R_r$, then $\chi(a, b)=r$;
\item if $(a, b)\in R_l\times R_r$ and $P_1(a, b)$, then $\chi(a, b)=P_1$;
\item if $(a, b)\in R_l\times R_r$ and $P_2(a, b)$, then $\chi(a, b)=P_2$;
\item if $(a, b)\in R_l\times R_r$ and $P_3(a, b)$, then $\chi(a, b)=P_3$.
\end{itemize}
 and a vertex coloring $\phi: \Gamma_i\longrightarrow\{L, R\}$ for every $a\in \Gamma_i$ by
 \begin{itemize}
 \item if $a\in R_l$, then $\phi(a)=L$;
 \item if $a\in R_l$, then $\phi(a)=R$.
 \end{itemize}

 Let $(\Gamma_i, \chi, \phi)$ be the $\alpha$-pattern $P$. By the Ne\v{s}et\v{r}il-R\"{o}dl Theorem, there is some bipartite graph $B_i\subset \Gamma$ such that for every partition $F$ on $B_i$, there is $\Gamma_i'\subseteq B_i$ such that
 \begin{enumerate}
 \item  $\Gamma_i'$ has the $\alpha$-pattern $P$ (hence $\Gamma_i'\cong\Gamma_i$);
 \item  $\Gamma_i'$ is $F$-homogeneous.
 \end{enumerate}

Now we choose $N\in \mathbb{N}$ such that when $j\geq N$, $\Gamma_j$ has all colors. Now let $g_i\in G$ be such that $g_i(B_i)$ has no cross-edge with cross-type $c$ and let $F=\chi\circ g_i$. Then $F=\chi\circ g_i: [B_i]^2\longrightarrow \{l, r, P_1, P_2, P_3\}\backslash\{c\}$. Since $g(\Gamma_i')$ has no cross-edge with cross-type $c$, the coloring  $\chi$ is not a permutation of $\chi\circ g_i$: so by Claim \ref{homogeneous} there must be distinct cross-types $P_m$ and $P_n$ and some cross-type $P_l$  such that for every $(a, b)\in (R_l\times R_r)\cap \Gamma_i'$, $\chi(x, y)=P_m$ or $\chi(x, y)=P_n$ implies $\chi\circ g_i(x, y)=P_l$, i.e.  $P_m(x, y)$ or $P_n(x, y)$ implies $P_l(g_i(x), g_i(y))$. WLOG, let $\Gamma_i$ replace $\Gamma_i'$.

Let $X, Y\subset\Gamma$ be finite bipartite subgraphs with $|X\cap R_i|=|Y\cap R_i|$ ($i=l, r$) and $f: X\longrightarrow Y$ such that $f$ preserves $P_k, R_l$ and $R_r$ where $k\neq m, n$. For $X$, there are some $N_x\in \mathbb{N}$ such that $\Gamma_{N_x}\supset X$. By a similar argument as above, there is some $g_X\in G$ such that for every $(a_x, b_x)\in X\cap (R_l\times R_r)$ with $P_m(a_x, b_x)\vee P_n(a_x, b_x)$, we have $P_l(g_X(a_x), g_X(b_x))$. Similarly, For $Y$, there is some $N_y\in \mathbb{N}$ such that $\Gamma_{N_y}\supset Y$. By a similar argument as above, there is some $g_Y\in G$ such that for every $(a_y, b_y)\in Y\cap (R_l\times R_r)$ with $P_m(a_y, b_y)\vee P_n(a_y, b_y)$, $P_l(g_Y(a_y), g_Y(b_y))$. Thus there is an isomorphism $\sigma: g_X(X)\longrightarrow g_Y(Y)$. Hence $g_Y \circ f=\sigma\circ g_X$, hence $f=g_Y^{-1}\circ \sigma\circ g_X$ and then $f\in G\upharpoonright X$. Since $X$ and $Y$ are arbitrary finite bipartite subgraphs of $\Gamma$ and $G$ is closed, so for any $f\in Sym_{\{l, r\}}(\Gamma)$ preserving $P_k$ for some $k\in \{1, 2, 3\}$, $f\in G$. By Definition \ref{irre}, $G$ is reducible, a contradiction with our assumption.

\textbf{Property (2)}:   Suppose there is no such $(R_1, \alpha)$ in $\Gamma$, i.e.
for any finite bipartite $R\subset \Gamma$ and any
$\alpha\in{(\Gamma\backslash R)\cap R_l}$, there exist some
$\sigma\in S_3$ and $g\in G$ such that $g\upharpoonright {R\cup
\{\alpha\}}$ is a switch w.r.t. $\alpha$ according to $\sigma$, but there is no
$f\in G$ such that $f$ is a switch w.r.t. $\alpha$
according to $\sigma$. Let $\Gamma=\cup \Gamma_i$ as in the SFSP, then there exists a sequence
$\{f_i, \sigma_i\}\subset G$ such that $f_i\upharpoonright (\Gamma_i\cup\{ \alpha\})$
is a switch w.r.t. a single vertex $\alpha$ of $R_l$ according to $\sigma_i$. Since
$S_3$ is finite  but $\{f_i\}$ is infinite, we have $\{f_{i_j}\}\subseteq \{f_i\}$
and some $\sigma$ such that $f_{i_j}\upharpoonright (\Gamma_i\cup\{ \alpha\})$ is
a switch w.r.t. a single vertex $\alpha$ of $R_l$ according to $\sigma$. Since $G$ is closed, there exists a switch w.r.t. $\alpha$ according to $\sigma$ is also in
$G$. But this contradicts the assumption.

\textbf{Property (3)}: Similar to the previous proof of \textbf{Property (2)}.

 Now we choose $R = R_0 \cup R_1 \cup R_2$. By Claim \ref{S}, the set $R$ is sufficiently complex.
\end{proof}


\begin{lem}\label{redu}
Suppose $H_1, H_2\leq S_3$.  If $G=\langle
S_{\{l\}}^{H_1}, S_{\{r\}}^{H_2}\rangle$ is an irreducible closed subgroup in $Sym(\Gamma)$, then for any $f\in H_1$ and any $g\in H_2$ it is the case that $f\circ g=g\circ f$.
\end{lem}

\begin{proof}
 Suppose not, then there exist $f\in H_1$, $g\in H_2$ such
that $f \circ g\neq g\circ f$. Then for
$\gamma=g^{-1}\circ f^{-1}\circ g\circ f$, there
must be some $c\in \{1, 2, 3\}$ such that $\gamma(c)\neq c$.
Choose $x\in R_l$, $y\in R_r$ such that $P_c(x, y)$. We construct an
$h\in G=\langle S_{\{l\}}^{H_1}, S_{\{r\}}^{H_2} \rangle$ which will be a composition
$h = g_4 \circ g_3 \circ g_2 \circ g_1$ of four switches $g_1, g_3\in S_{\{l\}}^{H_1}$, $g_2, g_4\in S_{\{r\}}^{H_2}$ on
single vertices. First let $g_1$ be a switch w.r.t. $x$   according to $f$, then let $g_2$ be a switch
w.r.t. $g_1(y)$  according to $g$.
Then let $g_3$ be a switch w.r.t. $g_2g_1(x)$   according to $f^{-1}$ and finally let $g_4$ be a
switch w.r.t. $g_3g_2g_1(y)$   according to
$g^{-1}$. Then for every $(a, b)\in R_l\times R_r$, if $(a,
b)\neq (x, y)$, then $P_c(a, b)\Longrightarrow P_c(h(a), h(b))$ but
$P_c(x, y)\Longrightarrow \neg P_c(h(x), h(y))$. Hence for any finite
bipartite $A\subset \Gamma$, we can construct a $g_A\in G$
such that $g_A(A)$ has one fewer edge with cross-type $c$. By repeating this process, we can find some $\overline{g}\in G$ such that
$\overline{g}(A)$ has no edge with cross-type $c$. Then $\Gamma$ cannot contain a sufficiently complex set. By Theorem \ref{existstence}, $G$ is reducible, a contradiction. This completes the proof of Lemma \ref{redu}.
\end{proof}

In particular we have:

$\langle S_{\{l\}}^{<(12)>}, S_{\{r\}}^{<(123)>}\rangle$ is reducible since
$(12)(123)=(23)\neq (13)=(123)(12)$.

$\langle S_{\{l\}}^{<(13)>}, S_{\{r\}}^{<(123)>}\rangle$ is reducible since
$(13)(123)=(12)\neq (23)=(123)(13)$.

$\langle S_{\{l\}}^{<(23)>}, S_{\{r\}}^{<(123)>}\rangle$ is reducible since
$(23)(123)=(13)\neq (12)=(123)(23)$.

$\langle S_{\{l\}}^{<(12)>}, S_{\{r\}}^{<(13)>}\rangle$ is reducible since
$(12)(13)=(132)\neq (123)=(13)(12)$.

$\langle S_{\{l\}}^{<(12)>}, S_{\{r\}}^{<(23)>}\rangle$ is reducible since
$(12)(23)=(123)\neq (132)=(23)(12)$.

$\langle S_{\{l\}}^{<(13)>}, S_{\{r\}}^{<(23)>}\rangle$ is reducible since
$(23)(13)=(123)\neq (132)=(13)(23)$.

\begin{lem}
The group  $S_{\{l,
r\}}^{S_3}=\langle S_{\{l\}}^{S_3}, S_{\{r\}}^{S_3}\rangle$ is the full symmetric group $Sym_{\{l, r\}}(\Gamma)$.
\end{lem}

\begin{proof}
It is enough to show that for any $(n\times m)$-bipartite
$A\subset\Gamma$ where $n, m\in \mathbb{N}$, there exists some $g_A\in
S_{\{l, r\}}^{S_3}$ such that $g_A(A)$ has  only a single cross-type $P_1$. Then for any two $(n\times m)$-bipartite graphs $B,
C$, we can find an automorphism $\sigma$ of $\Gamma$ sending $g_B(B)$ to $g_C(C)$, since each of these two subgraphs has only $P_1$ as cross-type. Then the map $f = g^{-1}_C \circ \sigma \circ g_B$ takes $B$ to $C$, and $f \in S_{\{l, r\}}^{S_3}$.
Then $S_{\{l, r\}}^{S_3}=Sym_{\{l, r\}}(\Gamma)$.

WLOG, suppose $A$ has three cross-types: $P_1, P_2$ and $P_3$. Let $f, g\in S_3$, $f=(123)$ and $g=(12)$. Then $f\circ g\neq g\circ f$ and let $\gamma=g^{-1}\circ f^{-1}\circ g\circ f$ ($=(123)$). Using the similar argument as in
the proof of Lemma \ref{redu}, for every finite bipartite subgraph $A\subset\Gamma$ we can construct
some $g\in S_{\{l, r\}}^{S_3}$ such that
$g(A)$ has no cross-edge with cross-type $P_2$. Similarly, using the similar argument as in
the proof of Lemma \ref{redu}, we can
construct some $f\in S_{\{l, r\}}^{S_3}$ such that $f(g(A))$ has no cross-edge with cross-type $P_3$. That is, for every
finite bipartite subgraph $A\subset\Gamma$, there exists $f\circ g\in S_{\{l, r\}}^{S_3}$ such that $h(A)$ has only a single cross-type $P_1$. This completes the proof of this Lemma.
\end{proof}

Then the candidates for nontrivial irreducible closed subgroups are:
\begin{itemize}
\item $S_{\{l\}}^{<(12)>}$, $S_{\{r\}}^{<(12)>}$ and $S_{\{l, r\}}^{<(12)>}(=\langle S_{\{l\}}^{<(12)>}, S_{\{r\}}^{<(12)>}\rangle)$;
\item $S_{\{l\}}^{<(13)>}$, $S_{\{r\}}^{<(13)>}$ and $S_{\{l, r\}}^{<(13)>}(=\langle S_{\{l\}}^{<(13)>}, S_{\{r\}}^{<(13)>}\rangle)$;
\item $S_{\{l\}}^{<(23)>}$, $S_{\{r\}}^{<(23)>}$ and $S_{\{l, r\}}^{<(23)>}(=\langle S_{\{l\}}^{<(23)>}, S_{\{r\}}^{<(23)>}\rangle)$;
\item $S_{\{l\}}^{<(123)>}$, $S_{\{r\}}^{<(123)>}$ and $S_{\{l, r\}}^{<(123)>}(=\langle S_{\{l\}}^{<(123)>},
S_{\{r\}}^{<(123)>}\rangle)$;
\item $S_{\{l\}}^{S_3}$ and $S_{\{r\}}^{S_3}$.
\end{itemize}


Note that when $\sigma\in S_3$ is nontrivial, $S_{\{l\}}^{<\sigma>}=\overline{\langle Aut(\Gamma),
h\rangle}$ where $h$ is a switch w.r.t. some subset of $R_l$ according to $\sigma$.

\begin{lem}\label{H12}
If $H_1\neq H_2$ are non-trivial subgroups of $S_3$, then $\langle S_{\{l\}}^{H_1}, S_{\{l\}}^{H_2}\rangle=S_{\{l\}}^{S_3}$.
\end{lem}

\begin{proof}
Let $H_1=<\sigma_1>$ and
$H_2=<\sigma_2>$ for $\sigma_1, \sigma_2\in S_3$, $\sigma_1\neq\sigma_2$ and $\sigma_1, \sigma_2\neq (1)$. There exist $f_1\in S_{\{l\}}^{H_1}$ which is a switch w.r.t. some vertex in $R_l$ according to $\sigma_1$ such that $S_{\{l\}}^{H_1}=\overline{\langle Aut(\Gamma), f_1\rangle}$, and $f_2\in S_{\{l\}}^{H_2}$ which is a switch w.r.t. some vertex in $R_l$ according to $\sigma_2$ such that $S_{\{l\}}^{H_2}= \overline{\langle Aut(\Gamma), f_2\rangle}$. Note that every two distinct proper subgroups of $S_3$   generate the whole $S_3$. Then every switch w.r.t. a single vertex in $R_l$ according to $\sigma\in S_3$ is generated by  the elements in $Aut(\Gamma)$ together with two additional elements  $f_1$ and $f_2$. Thus $\langle S_{\{l\}}^{H_1}, S_{\{l\}}^{H_2}\rangle= \overline{\langle Aut(\Gamma),
f_1, f_2\rangle}$ is the closed group generated by $Aut(\Gamma)$ and all the switches w.r.t. single vertex in $R_l$ according to some $\sigma\in S_3$, which is $S_{\{l\}}^{S_3}$ by definition.
\end{proof}

\begin{lem}\label{S_1}
Let $S=S_{\{l\}}^{<(12)>}(\Gamma)\cup S_{\{l\}}^{<(13)>}(\Gamma)\cup S_{\{l\}}^{<(23)>}(\Gamma)\cup S_{\{l\}}^{<(123)>}$, and $\sigma\in S_{\{l\}}^{S_3}(\Gamma)\backslash S$, then
$\overline{\langle Aut(\Gamma), \sigma\rangle}= S_{\{l\}}^{S_3}(\Gamma)$.
\end{lem}

\begin{proof}
Recall that $S_{\{l\}}^H(\Gamma)$ for $H\leq S_3$ is generated by compositions of switches on singletons in $R_l$ together with automorphisms. We will show that any element which is in $\in S_{\{l\}}^{S_3}(\Gamma)\backslash S$ can be modified to produce elements in two distinct subgroups $S_{\{l\}}^{H_1}(\Gamma)$ and $S_{\{l\}}^{H_1}(\Gamma)$ where $H_1\neq H_2$. Hence we can get all of $ S_{\{l\}}^{S_3}(\Gamma)$ as closure. Let $\sigma\in S_{\{l\}}^{S_3}$ but not in any of $S_{\{l\}}^{H}$ for every $H<S_3$.
By considering the action of $\sigma$ on each edge, we can find sets  $\{A_1, A_2, A_3, A_4, A_5\}$ where $A_i\subseteq R_l$ for $1\leq i\leq 5$ and $A_j\cap A_k=\emptyset$ for $j\neq k$ ($A_i$ could be empty for $1\leq i\leq 5$, but there are at least two distinct $j, k\in\{1, 2, 3, 4, 5\}$ such that $A_j$ and $A_k$ are nonempty) such that $\sigma$ sends the cross-types on the cross-edges with exactly one endpoint in $A_1$ to the  cross-types dictated by $(12)$, the  cross-types on the cross-edges with exactly one endpoint in $A_2$ change the cross-types according to $(13)$, the  cross-types on the cross-edges with exactly one endpoint in $A_3$ change the  cross-types according to $(23)$, the  cross-types on the cross-edges with exactly one endpoint in $A_4$ change the  cross-types according to $(123)$ and the  cross-types on the cross-edges with exactly one endpoint in $A_5$ change the  cross-types according to $(132)$.

\textbf{Case 1} Assume $0<|A_i|<\omega$ for $i=1, \dots, 5$.

Step $1$: There exists some
$f_1\in Aut(\Gamma)$ such that $f_1(\sigma(A_i))=A_i$. Let
$h_1=\sigma\circ f_1\circ \sigma$, then $h_1$ is a switch w.r.t. $A_4$
according to $(132)$ and w.r.t. $A_5$ according to $(123)$.

Step $2$: Similarly, there exists some $f_2\in Aut(\Gamma)$ such that
$f_2(\sigma(A_i))=A_i$ for $i=1, \dots, 5$, and let $h_2=h_1\circ f_2$, then $h_2$ is a
switch w.r.t. $\sigma(A_4)$ by $(132)$ and w.r.t. $\sigma(A_5)$ by
$(123)$. Let $h_3=h_2\circ \sigma$, then $h_3$ is a switch w.r.t.
$A_1$ according to $(12)$, w.r.t. $A_2$ according to $(13)$ and
w.r.t. $A_3$ according to $(23)$.

Step $3$: There exists $f_3\in Aut(\Gamma)$ such that $f_3(A_4)\cap A_1\neq\emptyset$ but $f_3(A_4)\cap A_i=\emptyset$ for $i=2, 3$, and $f_3(A_5)\cap A_2\neq\emptyset$ but  $f_3(A_5)\cap A_i=\emptyset$ for $i=1, 3$ and $f_3(A_5)\cap f_3(A_4)=\emptyset$. There exists $f_4\in Aut(\Gamma)$ such that $f_4(h_3\circ f_3(A_j))=A_j$ for $j=4, 5$ and $f_4\circ h_3(A_1)\cap A_5\doteq\emptyset, f_4\circ h_3(A_5)\cap A_4\doteq\emptyset$. Let $h_4=h_1\circ f_4\circ h_3$. By the definition of $h_3, f_4, h_1$ and the fact that $(132)(12)=(23)$ and $(123)(13)=(23)$. Now $h_4$ is a switch w.r.t. $A_1\backslash f_3(A_4)\subset A_1$ according to $(12)$, w.r.t. $A_2\backslash f_3(A_5)\subset A_2$ according to $(13)$, and w.r.t. $f_3(A_4)\backslash A_1$ according to $(132)$, w.r.t. $f_3(A_5)\backslash A_2$ according to $(123)$, and w.r.t. $A_3\cup (A_1\cap f_3(A_4))\cup (A_2\cap f_3(A_5))\supset A$ according to $(23)$.


Step $4$: Now we let $A_1^1=A_1\backslash f_3(A_4)$, $A_2^1=A_2\backslash f_3(A_5)$, $A_3^1=A_3\cup (A_1\cap f_3(A_4))\cup (A_2\cap f_3(A_5))$, and $A_4^1=f_3(A_5)\backslash A_2$, $A_5^1=f_3(A_4)\backslash A_1$, and let $h_4=\sigma_1$. Note that $A_1^1\subset A_1, A_2^1\subset A_2$. Then $\sigma_1$ is a switch w.r.t.t $A_1^1$ according to $(12)$, w.r.t. $A_2^1$ according to $(13)$, w.r.t. $A_3^1$ according to $(23)$, w.r.t. $A_4^1$ according to $(123)$ and w.r.t. $A_5^1$ according to $(132)$. Now follow Step $1-3$, we can get a switch w.r.t. $A_1^1$ according to $(12)$, w.r.t. $(A_2^1)$ according to $(13)$ and w.r.t. $(A_3^1)$ according to $(23)$.

Step $5$: Since $|A_i|<\omega$ for $i=1, 2$, we can follow steps $1-4$ finitely many times to obtain a sequence of $A_1^i, A_2^i, A_3^i\subset\Gamma$, ending with  $A_1^N=\emptyset$ and $A_2^N=\emptyset$. Then we now have a switch $g$ w.r.t. $A_3^N$ according to $(23)$.

Similarly, we  produce a switch $h$ w.r.t. some subset of $R_l$ according to $(13)$.  Then $\overline{\langle Aut(\Gamma),
\sigma\rangle}=\overline{\langle Aut(\Gamma), g, h,
\sigma\rangle}=\langle S_{\{l\}}^{<(23)>}, S_{\{l\}}^{<(13)>}\rangle$. By Lemma \ref{H12} $\overline{\langle Aut(\Gamma), g, h,
\sigma\rangle}=\langle S_{\{l\}}^{<(23)>}, S_{\{l\}}^{<(13)>}\rangle= S_{\{l, r\}}^{S_3}$, and this completes the proof of \textbf{Case 1}.

\textbf{Case 2} If all $A_k$'s are finite but some $A_k$ is  empty, then we   follow the proof   above, only with  fewer steps.

\textbf{Case 3} If there exists some $k$ with  $|A_k|=\omega$, then since $A_k$ is
the ``limit'' of its finite approximations $A_k'$,  we can deal
with  its finite subsets $A_k'$ with the method in \textbf{Case 1}, then take the
limit (which must lie in the closure).

This completes the proof of Lemma \ref{S_1}.
\end{proof}


\begin{thm}\label{Gin}
Let $G$ be an irreducible closed subgroup of $S_{\{l\}}^{S_3}(\Gamma)$ containing $Aut(\Gamma)$. Then there exists $H\leq S_3$
such that $G=S_{\{l\}}^{H}(\Gamma)$.
\end{thm}
\begin{proof}
Let $G$ be an irreducible closed proper subgroup of $S_{\{l\}}^{S_3}(\Gamma)$ containing $Aut(\Gamma)$. If there is some nontrivial $H\leq S_3$ such that $G\subseteq S_{\{l\}}^H(\Gamma)$, then if $G$ is not $Aut(\Gamma)$, we have $G=S_{\{l\}}^{H}(\Gamma)$. Suppose not, i.e. $G\subset S_{\{l\}}^{H}(\Gamma)$. Then every $g\in G\backslash Aut(\Gamma)$ is a composition of switches w.r.t. permutations in $H$. If $|H|=2$, then $g$ is also a switch w.r.t. some subset $A\subseteq R_l$ according to some $\sigma\in H$. Then by Lemma \ref{SSAlG} $G\geq \overline{\langle Aut(\Gamma), g\rangle}=S_{\{l\}}^H(\Gamma)$, contradicting the assumption that $G\subset S_{\{l\}}^{H}(\Gamma)$. If  $|H|= 3$, we can use an argument similar to that in Lemma \ref{S_1} to get  some $g'\in\overline{ \langle Aut(\Gamma), g\rangle}$ such that $g'\in G\backslash Aut(\Gamma)$ is a switch w.r.t. some subset $A\subseteq R_l$ according to some $\sigma\in H$. Then by Lemma \ref{SSAlG} $G\geq \overline{\langle Aut(\Gamma), g'\rangle}=S_{\{l\}}^H(\Gamma)$, contradicting with the assumption that $G\subset S_{\{l\}}^{H}(\Gamma)$.

If such $H$ does not exist, then there exists some element $f\in G$ but $f\notin S_{\{l\}}^{K}(\Gamma)$ for any proper $K\leq S_3$. Then by Lemma \ref{S_1}, there exist at least two nontrivial $H_1, H_2\leq S_3$ with $H_1\neq H_2$ such that $f\in \langle S_{\{l\}}^{H_1}(\Gamma),S_{\{l\}}^{H_2}(\Gamma)\rangle$, thus $G\geq S_{\{l\}}^{S_3}(\Gamma)$, a contradiction.

\end{proof}

\section{Switch Groups as Irreducible Closed Subgroups}
Now we show that the the various switch groups are exactly the irreducible closed subgroups of $Sym(\Gamma)$ where $\Gamma$ is the random bipartite graph having three cross-types.

\begin{lem}\label{Slrgroups}
Let $G$ be an irreducible closed subgroup of $Sym_{\{l, r\}}(\Gamma)$ containing $Aut(\Gamma)$, and suppose there exist a finite bipartite subgraph $R\subset\Gamma$ and
$\alpha\in \Gamma\backslash R$   such that
\renewcommand{\labelenumi}{(\arabic{enumi})}
\begin{enumerate}
\item $(R, \alpha)$ is a sufficiently complex subgraph with respect to $G$,
\item for any $\pi\in G$ such that $\pi\upharpoonright R$ is an
isomorphism, $\pi\upharpoonright{R\cup\{\alpha\}}$ is a
switch w.r.t. $\alpha$ according to some $\sigma\in S_3$ ($\sigma$ could be the identity).
\end{enumerate}

Then $G=\langle S_{\{l\}}^{H_1}(\Gamma), S_{\{r\}}^{H_2}(\Gamma)\rangle$ for some $H_1, H_2\leq
S_3$. If $G\subset Sym_{\{l, r\}}(\Gamma)$, then $H_1=H_2$ unless one of the two groups is trivial.
\end{lem}

\begin{proof}
Since $G\cap S_{\{l\}}^{S_3}(\Gamma)$ is an irreducible closed subgroup of $S_{\{l\}}^{S_3}(\Gamma)$, by Theorem \ref{Gin} there exists $H_1\leq S_3$ such that $S_{\{l\}}^{H_1}(\Gamma)=G\cap S_{\{l\}}^{S_3}(\Gamma)$. Similarly, there exists $H_2\leq S_3$ such that $S_{\{r\}}^{H_1}(\Gamma)=G\cap S_{\{r\}}^{S_3}(\Gamma)$. Now let $\pi\in G$ be given. Choose $N$ large enough such that for our fixed enumeration
$\Gamma=\cup\Gamma_i$, if $i\geq N$  then
\renewcommand{\labelenumi}{(\arabic{enumi})}
\begin{enumerate}
\item $\Gamma_i$ is sufficiently complex,
\item for any $x, y, z\in \Gamma_i$ with $x$ on the same side as $\alpha$ and $x\neq y, z$, there exists subgraphs $R_y$ and $R_z$ of $\Gamma_i$ with $y\in R_y, z\in R_z, x\notin R_y\cup
R_z, (R_y, x)\cong (R, \alpha)\cong (R_z, x)$, and with cross-edges between $x$ and $R_y\cap R_z$ of all three cross-types.
\end{enumerate}
We can find an $N$ such that $(1)$ holds by Theorem
\ref{existstence} and SFSP. By the extension property of $\Gamma$, $(2)$
holds in $\Gamma$, and hence by the Strong Finite Submodel Property,
$(2)$ holds for $\Gamma_i$  for all large $i$.

Now look at $\Gamma_N$, and define the coloring $\chi$ on $[\Gamma_N]^{\leq2}$ by
\begin{enumerate}
\item for every $x\in \Gamma_N$,
 \begin{itemize}
 \item if $x\in R_l$, then $\chi(x)=L$;
 \item if $x\in R_r$, then $\chi(x)=R$.
 \end{itemize}
\item for every $\{a, b\}\in [\Gamma_N]^{2}$,
\begin{itemize}
\item if $\{a, b\}\in [R_l]^2$, then $\chi(a, b)=l$;
\item if $\{a, b\}\in [R_r]^2$, then $\chi(a, b)=r$;
\item if $(a, b)\in R_l\times R_r$ and $P_1(a, b)$, then $\chi(a, b)=P_1$;
\item if $(a, b)\in R_l\times R_r$ and $P_2(a, b)$, then $\chi(a, b)=P_2$;
\item if $(a, b)\in R_l\times R_r$ and $P_3(a, b)$, then $\chi(a, b)=P_3$.
\end{itemize}
\end{enumerate}

Let $(\Gamma_N, \chi)$ be the $\alpha$-pattern P. By the
Ne\v{s}et\v{r}il-R\"{o}dl theorem, there exists a $\alpha$-Pattern $Q$, with the underlying set $X$, such that for any partition $F: [X]^2\longrightarrow \{P_1, P_2, P_3, l, r\}$, there exists $\bar{\Gamma_N}\subset X$ such that
\begin{itemize}
\item $(\bar{\Gamma_N}, \chi)$ has the $\alpha$-pattern P (hence $\bar{\Gamma_N}\cong \Gamma_N$);
\item $(\bar{\Gamma_N}, \chi)$ is $F$-homogeneous.
\end{itemize}
We define $F=\chi\circ\pi$ on $[\Gamma]^2$ for every $\{a, b\}\in [\Gamma]^2$ by
\begin{itemize}
\item if $\{a, b\}\in [R_l]^2$, then $F(a, b)=l$;
\item if $\{a, b\}\in [R_r]^2$, then $F(a, b)=r$;
\item if $P_1(\pi(a), \pi(b))$, then $F(a, b)=P_1$;
\item if $P_2(\pi(a), \pi(b))$, then $F(a, b)=P_2$;
\item if $P_3(\pi(a), \pi(b))$, then $F(a, b)=P_3$;
\end{itemize}
Since $(\bar{\Gamma_N}, \chi)$ is $F$-homogeneous, $\chi$ is homogeneous w.r.t. $\chi\circ\pi$. Since $\Gamma_N$ is sufficiently complex and $\bar{\Gamma_N}\cong\Gamma_N$, $\bar{\Gamma_N}$ witnesses all the cross-types. Hence $\chi\circ\pi$ does not lose any
cross-types. By Claim \ref{homogeneous}, we have $\chi\circ\pi$
is a permutation of $\chi$. So $\pi$ is a switch w.r.t. all the
vertices in $R_l$ according to some $\sigma\in S_3$. Since $\pi\in G\cap S_{\{l\}}^{S_3}(\Gamma)$, such $\sigma\in H_1$. Hence $\pi\in
S_{\{l\}}^{H_1}(\Gamma)\upharpoonright \bar{\Gamma_N}$.

Since $\Gamma_N$ is sufficiently complex, it witnesses which
switches are in $G$, so there must be some $\phi\in G\cap
S_{\{l\}}^{H_1}(\Gamma)$ such that $\phi^{-1}\circ\pi$ is an isomorphism
on $\Gamma_N$. Let $\pi_1=\phi^{-1}\circ\pi$;  we show that
$\pi_1$ is a switch map.

Write $\Gamma_{N+1}=\Gamma_N\cup\{x\}$. WLOG, we suppose $x\in R_l$.
If $\pi_1\upharpoonright \Gamma_{N+1}$ is an isomorphism, then it is
trivially a switch. If not, then let $y, z\in \Gamma_N$ be arbitrary. By
the choice of $N$ there exists subgraphs $y\in R_y\subset\Gamma_N$
and $z\in R_z\subset\Gamma_N$ such that there are all cross-types
between $\{x\}$ and $R_y\cap R_z$. $\pi_1\upharpoonright
\Gamma_N$ is an isomorphism, and $R_y, R_z, (R_y, x), (R_z, x)$ are
sufficiently complex with $R_y\cong R$ and $R_z\cong R$. Since $\pi_1\upharpoonright R_y$ is an isomorphism,
$\pi_1\upharpoonright R_y\cup\{x\}$ is a switch w.r.t. $x$ according to some $\sigma_1\in H_1$. Similarly, $\pi_1\upharpoonright R_z\cup\{x\}$ is a
switch w.r.t. $x$ according to some $\sigma_2\in H_1$. But $\sigma_1=\sigma_2$ since
$\pi_1\upharpoonright R_y\cup\{x\}$ has to agree with
$\pi_1\upharpoonright R_z\cup \{x\}$ on $R_y\cap R_z$. Since $y, z$
are arbitrary, $\pi_1\upharpoonright \Gamma_{N+1}$ is a switch w.r.t. $x$ according to some $\sigma'\in S_3$. $(R_y, x)$ witnesses the fact that this switch is in $G$, so
there exists some $\phi_1\in G$ such that $\phi_1$ is a switch w.r.t. $x$ according to $\sigma'$. Since $\phi_1\in  G\cap S_{\{l\}}^{S_3}(\Gamma)$, such $\sigma'\in H_1$. Hence $\phi_1^{-1}\circ\pi_1$ is an isomorphism on
$\Gamma_{N+1}$. Let $\pi_2=\phi_1^{-1}\circ\pi_1$.

Write $\Gamma_{N+2}=\Gamma_{N+1}\cup\{x'\}$ for $x'\in R_r$.
Similarly, we get $\pi_2\upharpoonright \Gamma_{N+2}$ is a switch
w.r.t. $x'$. So there exists some $\phi_2\in G$ and
$\phi_2\upharpoonright \Gamma_{N+2}$ is a switch w.r.t. $x'$ and
$\phi_2'\circ\pi_2$ is an isomorphism on $\Gamma_{N+2}$. By
induction, $\pi\upharpoonright \Gamma_{N+k}$ is a composition of
switches on vertices in $\Gamma_{N_k}\backslash \Gamma_N$. Since
$\langle S_{\{l\}}^{H_1}(\Gamma), S_{\{r\}}^{H_2}(\Gamma)\rangle$ is closed,
$\pi$ is a switch. We have shown that it is a composition of
switches on the vertices in $R_l$ and switches on the vertices
in $R_r$. Since the choice of $\pi$ is arbitrary, $G\subseteq
\langle S_{\{l\}}^{H_1}(\Gamma), S_{\{r\}}^{H_2}(\Gamma)\rangle$. But
$S_{\{l\}}^{H_1}(\Gamma), S_{\{r\}}^{H_1}(\Gamma)\leq G$, so we have $G=\langle
S_{\{l\}}^{H_1}(\Gamma), S_{\{r\}}^{H_2}(\Gamma)\rangle$. If $G\subset Sym_{\{l, r\}}(\Gamma)$,
then by Lemma \ref{H12}, $H_1=H_2$ unless one of them is trivial .
\end{proof}


\section{Switches Are the Only Nontrivial Closed Groups}
Now we show that there are no other irreducible closed subgroups.

\begin{lem}\label{lemmaA}
Let $\Gamma=\cup\Gamma_i$ as in SFSP, and let $G$ be an irreducible closed subgroup of $Sym_{\{l, r\}}(\Gamma)$ containing $Aut(\Gamma)$. There is an integer $N$ such that $i\geq N$ implies

$(\star)$ If $\alpha\in \Gamma_i$ and $g\in G$ is an isomorphism on
$\Gamma_i\backslash\{\alpha\}$, then there are cross-edges
between $g(\Gamma_i\backslash \{\alpha\})$ and $g(\alpha)$ with all
cross-types.
\end{lem}

\begin{proof}
Let $R$ be sufficiently complex. By the extension properties $\Theta_n$ for $n\in\mathbb{N}$, the
following sentence is true in $\Gamma$:

$(\ast)$ For any $\alpha\in \Gamma$ and any bipartite subgraph $B\subset\Gamma$ with $|B|=|R|$ and any
vertex $b\in B$ on the same side as $\alpha$, there is some
embedding $\phi: B\longrightarrow \Gamma$ such that
$\phi(b)=\alpha$.

By SFSP, there exists $N\in\mathbb{N}$ such that when $i\geq N$,
$(\ast)$ is true for $\Gamma_i$.

Suppose for any $M\in \mathbb{N}$, there exist $i\geq M$ and some $g\in
G$ which is an isomorphism on $\Gamma_i\backslash\{\alpha\}$, but for some
$c\in\{P_1, P_2, P_3\}$, there is no cross-edge in $g(\Gamma_i)$
with endpoint $g(\alpha)$ having the cross-type $c$. For any
subgraph $B\subset\Gamma$ with $|B|=|R|$ and any $b$ in the same side as $\alpha$, there is an
isomorphism $\phi\in Aut(\Gamma)$ with $\phi(B)\subseteq\Gamma_i$ such
that $\phi(b)=\alpha$. If $f=g\circ\phi\in G$, then $f(B)$
has no cross-edge with cross-type $c$ and endpoint $f(b)$. By a composition of such
maps, one for each vertex of $B$ in the same side as $\alpha$, we
have a $f^{\ast}$ such that $f^{\ast}(B)$ has no cross-type $c$.
Since $R$ is a special case of $B$, we can find a $f_0^{\ast}\in
G$ such that $f_0^{\ast}(R)$ has no cross-edge with cross-type $c$. But
$R$ is sufficiently complex, and witnesses the fact that $G$ is irreducible,
and we have reached a contradiction.
\end{proof}

\begin{lem}\label{lemmaB}
Let $\Gamma=\cup\Gamma_i$ given by SFSP, and let $G$ be an irreducible closed group of $Sym_{\{l, r\}}(\Gamma)$ containing $Aut(\Gamma)$. Then there is an integer $N$ such that for any
$i\geq N$:
\renewcommand{\labelenumi}{(\arabic{enumi})}
\begin{enumerate}
\item There is no $g\in G$ which is an isomorphism on $\Gamma_i$
except for its effect on one cross-edge.
\item For every $\alpha\in \Gamma$, there is no $g\in G$ which is an isomorphism on
$\Gamma_i\backslash\{\alpha\}$ and for which there exists a switch $f$
w.r.t. $g(\alpha)$ such that $f\circ g$ is an isomorphism on $\Gamma_i$
except for exactly one cross-edge.
\end{enumerate}
\end{lem}

\begin{proof}
Let $(C, \gamma)$ where $C\subset\Gamma$ and
$\gamma\in\Gamma\backslash C$ be sufficiently complex. Since
$\Gamma$ has the extension property, the following is true in
$\Gamma$:
\renewcommand{\labelenumi}{(\alph{enumi})}
\begin{enumerate}
\item for any $|C|-$graph $C'$ with three cross-types, and any
cross-edge $(a, b)$ of $\Gamma$, and any cross-edge $(a', b')\in C'$ with the same cross-type as $(a, b)$,
there is an embedding $\phi: C'\longrightarrow \Gamma$ such that
$\phi(a', b')=(a, b)$.
\item For any two vertices $\alpha,\beta$ on the same side of $\Gamma$ as is $\gamma$, there is a
finite subgraph $A\subset\Gamma$ such that $\alpha, \beta\notin A$ and
$A\cong C$ can be extended to $A\cup\{\alpha\}\cong C\cup\{\gamma\}$.
\end{enumerate}

By SFSP, there exists some $N\in\mathbb{N}$ such that $i\geq N$
implies that $\Gamma_i$ has the properties $(a)$ and $(b)$. We show
that the same $N$ will satisfy Lemma \ref{lemmaB}.

Since $(2)$ implies $(1)$ when a switch is trivial, we suppose
$(2)$ does not hold, i.e. for any $N\in\mathbb{N}$ there are $i\geq N, g\in G$, and $f$
such that $g$ is an isomorphism
on $\Gamma_i\backslash\{\alpha\}$, $f$ is a
switch w.r.t. $g(\alpha)$ and $f\circ g$ is an isomorphism
except on one cross-edge of $\Gamma_i$. One endpoint must be
$\alpha$, let the other endpoint be $\beta$. By $(b)$ above, there
exists some $A\cong C$ and $\alpha, \beta\notin A$ such that
$A\cup\{\alpha\}\cong C\cup\{\gamma\}$.

Now $g\upharpoonright A\cup\{\alpha\}$ is a switch with the same
permutation of cross-types as that of $f^{-1}$. But
$A\cup\{\alpha\}\cong C\cup\{\gamma\}$ which witnesses which
switches are in $G$. So $f^{-1}\in G$,
hence $f\in G$. Let $h=f\circ g$. Then $h\in G$, but this leads to a
contradiction since for any $|C|-$graph $C'$ having three cross-types, and any cross-edge $(\alpha',
\beta')\in C'\cap (R_l\times R_r)$ with $P_i(\alpha, \beta)$ and $P_i(\alpha', \beta')$ for some $i\in\{1, 2, 3\}$, by $(a)$ there is an embedding an embedding $\phi: C'\longrightarrow \Gamma$ such that
$\phi(\alpha', \beta')=(\alpha, \beta)$. Since $\Gamma$ is homogeneous, there is some $\Phi\in Aut(\Gamma)$ such that $\Phi\upharpoonright C'=\phi$. Now let $h_1=h\circ \Phi$. Then $h_1\in G$ and $h_1\upharpoonright C'$
is an isomorphism except on the cross-edge $(\alpha', \beta')$. Hence $h(C')$
has one fewer cross-edge with cross-type $P_i$ on it. If there are $n$ many cross-edges with cross-type $P_i$ in $C'$, then by repeating this argument $n$ times, we can construct $h_n\in G$ such that $h_n(C')$ has no cross-edge with cross-type $P_i$. Since $C$ is a special
case of $C'$, we have the same result for $C$. But since $C$ is sufficiently complex, it witnesses all the cross-types in $C$ and we have a
contradiction, and $N$ is as desired.
\end{proof}

Given a bipartite graph $A\subseteq\Gamma$, we define a class of vertex colorings $\Phi: A\longrightarrow \{L, \overline{P_1}, \overline{P_2}, \overline{P_3}\}$ for every $x\in A$ by
\begin{itemize}
\item if $x\in R_l$, then $\Phi(x)=L$;
\item if $x\in R_r$, then $\Phi(x)=\overline{P_i}$ for some $i=1, 2, 3$.
\end{itemize}
\begin{defn}\label{vpermutation}
Let $A_1$ be a bipartite graph with vertex coloring $\Phi_1$, and let $A_2$ be a bipartite graph with vertex coloring $\Phi_2$, where $\Phi_1$ is defined as above and $\Phi_2=\Phi\circ g$ for some $g\in Sym(A_1)$ preserving $R_l, R_r$. If $|A_1|=|A_2|$ and $|A_1\cap R_l|=|A_2\cap R_l|$, then the vertex coloring $\Phi_2$ is a \textit{permutation}
of the vertex coloring $\chi_1$ if there is some vertex bijection $\phi:
A_1\longrightarrow A_2$ preserving $R_l, R_r$ and some permutation $\sigma\in S_3$
such that for any vertex $x\in R_r$, $\Phi_1(x)=\sigma(\Phi_2(\phi(x)))$.
\end{defn}

\begin{defn}
Let $A$ be  a bipartite graph, and $\Phi_1, \Phi_2$  be vertex colorings on $A$ defined in Definition \ref{vpermutation}. Then the vertex coloring $\Phi_2$ is \textit{homogeneous} w.r.t. the coloring $\Phi_1$ if for any $x, x'\in A\cap R_r$, $\Phi_2(x)=\Phi_2(x')\Longrightarrow \Phi_1(x)=\Phi_1(x')$.
\end{defn}

\begin{claim}\label{vhomogeneous}
If $A$ is a bipartite graph, $\Phi_1, \Phi_2$ are the vertex colorings on $A$ defined as above, and $\Phi_2$ is homogenous w.r.t. $\chi_1$ but is not a permutation of $\Phi_1$, then there must be two distinct colors $\overline{P_i}$ and $\overline{P_j}$ and some color $\overline{P_k}$ ($i, j, k\in \{1, 2, 3\}$) such that for any $x\in A\cap R_r$, $\Phi_2(x)=\overline{P_i}$ or $\Phi_2(x)=\overline{P_j}$ implies $\Phi_1(x)=\overline{P_k}$.
\end{claim}

\begin{proof}
This follows immediately from the definitions  of homogeneous and permutation colorings.
\end{proof}

Given a bipartite graph $A\subseteq\Gamma$, we define an edge coloring $\chi: A\longrightarrow \{l, r, P_1, P_2, P_3\}$ for every $\{a, b\}\in A$ by
\begin{itemize}
\item if $\{a, b\}\subseteq R_l$, then $\chi(a, b)=l$;
\item if $\{a, b\}\subseteq R_r$, then $\chi(a, b)=r$;
\item if $a\in R_l, b\in R_r$ and $P_1(a, b)$, then $\chi(a, b)=P_1$;
\item if $a\in R_l, b\in R_r$ and $P_2(a, b)$, then $\chi(a, b)=P_2$;
\item if $a\in R_l, b\in R_r$ and $P_3(a, b)$, then $\chi(a, b)=P_3$.
\end{itemize}

\begin{lem}\label{lemmaC}
Let $\mathfrak{A}$ be the class of bipartite subgraphs having at least one vertex in $R_i$ ($i=l, r$) in $\Gamma$ with edge
coloring $\chi: [V]^2\longrightarrow \{l, r, P_1, P_2, P_3\}$ defined above, and with vertex coloring $\Phi: V\longrightarrow\{L, \overline{P_l}, \overline{P_2}, \overline{P_3}\}$ defined above. For any finite $A_1\in \mathfrak{A}$, there is a finite $A_2\in
\mathfrak{A}$ such that for any vertex coloring $\Psi:
A_2\longrightarrow \{L, \overline{P_1}, \overline{P_2}, \overline{P_3}\}$, which  is not a permutation of $\Phi$, then there exists
$A_1'\in\mathfrak{A}$ with $\langle A_1, \Phi, \chi\rangle\cong \langle A_1', \Phi, \chi\rangle$, $A_1'\subset A_2$, and one
of the following properties:
\renewcommand{\labelenumi}{(\alph{enumi})}
\begin{enumerate}
\item There is a color $t\in\{\overline{P_1}, \overline{P_2}, \overline{P_3}\}$ such that $\Psi(x)\neq
t$ for every vertex $x\in A_1'\subset A_2$.
\item There is some vertex coloring $\Psi': A_1'\longrightarrow \{L, \overline{P_1}, \overline{P_2},
\overline{P_3}\}$, which is a permutation of $\Psi$, and differs from $\Phi:
A_1'\longrightarrow\{L, \overline{P_1}, \overline{P_2}, \overline{P_3}\}$ on exactly one vertex in $R_r$.
\end{enumerate}
\end{lem}
\begin{proof}

We choose $M\in\mathbb{N}$ such that for any $i\geq M$, and for any
$v\in\Gamma_i$, there exists $B\subset \Gamma_i$ with $v\in B$ such
that $B\cong A_1$. We can do this because of the extension property
and SFSP. Now we choose $N\in \mathbb{N}$ such that $N\geq M$ and $\Gamma_N\supseteq A_1$. Since for any
$v\in\Gamma_N$, there exists an isomorphic copy of $A_1$ containing $v$, we can extend the colorings $\Phi, \chi$ on $A_1$ to the colorings on the whole $\Gamma_N$. We call them $\Phi, \chi$ to simplify the notation. Now let $\langle \Gamma_N, \Phi, \chi \rangle$ be our $P-$pattern.
Then by the  Ne\v{s}et\v{r}il-R\"{o}dl Theorem, there exists a finite $A_2\in
\mathfrak{A}$ such that for any partition function $F=\Psi: A_2\longrightarrow \{L, \overline{P_1}, \overline{P_2}, \overline{P_3}\}$,
there is $\Gamma_N'\subset A_2$ satisfying
\begin{itemize}
\item $\Gamma_N'$ has the $\alpha$-pattern $P$ (hence $\Gamma_N' \cong \Gamma_N$);
\item $(\Gamma_N', \Phi)$ is $\Psi$-homogeneous.
\end{itemize}
Since $\Gamma_N'\cong \Gamma_N$,  we can find a graph $A_1'\subset
\Gamma_N'$ such that $\langle A_1', \Phi, \chi\rangle\cong \langle A_1, \Phi, \chi\rangle$. Now we have $\langle\Gamma_N',
\Phi\rangle$ is $\Psi-$homogeneous, i.e. if $\Phi(a)=\Phi(b)$, then
$\Psi(a)=\Psi(b)$ for $a, b\in \Gamma_N'$.

If $\Psi$ is not a permutation of $\Phi$ on $\Gamma_N'$, then by Claim \ref{vhomogeneous} there
must exist some color $t\in \{\overline{P_1}, \overline{P_2}, \overline{P_3}\}$ such that $\Psi(v)\neq t$ for
every $v\in \Gamma_N'$. Since $A_1'\subset\Gamma_N'$, the same
result holds for $A_1'$. Then Property $(1)$ holds.

If $\Psi$ is a permutation of $\Phi$ on $\Gamma_N'$, then since $\Psi$ is not a permutation of $\Phi$ on $A_2\supset \Gamma_N'$, there must
exist some $\Gamma_k$ for $k\geq N$ such that $\Psi$ is a permutation
of $\Phi$ on $\Gamma_k$, but $\Psi$ is not a permutation of $\Phi$
on $\Gamma_{k+1}$. Let $\Gamma_{k+1}=\Gamma_k\cup \{v\}$. Note that $v\in R_r$ since $\Psi(x)=\Phi(x)$ for every $x\in R_l$. There must
exist a subgraph $B\subset \Gamma_{k+1}$ such that $v\in B$ and $f: B\cong A_1$
preserves the vertex coloring $\Phi$ and the edge coloring $\chi$. Hence we have $\Psi\upharpoonright (B\backslash \{v\})$ is a permutation
of $\Phi\upharpoonright (B\backslash \{v\})$, but
$\Psi\upharpoonright B$ is not a permutation of
$\Phi\upharpoonright B$. Let $f_1$ be an isomorphism from $A_1$ to $A_1'$, and
let $\Psi'=\Psi\circ f^{-1}\circ f_1^{-1}$. Then we have the Property $(2)$. This completes the proof of Lemma \ref{lemmaC}.
\end{proof}

Similarly, given a bipartite graph $A\subseteq\Gamma$, we define a class of vertex colorings $\overline{\Phi}: A\longrightarrow \{R, \overline{P_1}, \overline{P_2}, \overline{P_3}\}$ by for every $x\in A$,
\begin{itemize}
\item if $x\in R_r$, then $\Phi(x)=R$;
\item if $x\in R_l$, then $\Phi(x)=\overline{P_i}$ for some $i=1, 2, 3$.
\end{itemize}

A similar argument gives the following Lemma:
\begin{lem}\label{lemmaD}
Let $\mathfrak{A}$ be the class of bipartite subgraphs of $\Gamma$ having at least one vertex in $R_i$ ($i=l, r$) with edge
coloring $\chi: [V]^2\longrightarrow \{l, r, P_1, P_2, P_3\}$ defined above and with vertex coloring $\overline{\Phi}: V\longrightarrow\{R, \overline{P_l}, \overline{P_2}, \overline{P_3}\}$ defined above. For any finite $A_1\in \mathfrak{A}$, there is a finite $A_2\in
\mathfrak{A}$ such that for any vertex coloring $\Psi:
A_2\longrightarrow \{R, \overline{P_1}, \overline{P_2}, \overline{P_3}\}$, and which is not a permutation of $\overline{\Phi}$, there exists
$A_1'\in\mathfrak{A}$ with $\langle A_1, \overline{\Phi}, \chi\rangle\cong \langle A_1', \overline{\Phi}, \chi\rangle$, $A_1'\subset A_2$, and one
of the following properties:
\renewcommand{\labelenumi}{(\alph{enumi})}
\begin{enumerate}
\item There is a color $t\in\{\overline{P_1}, \overline{P_2}, \overline{P_3}\}$ such that $\Psi(x)\neq
t$ for every vertex $x\in A_1'\subset A_2$.
\item There is some vertex coloring $\Psi': A_1'\longrightarrow \{R, \overline{P_1}, \overline{P_2},
\overline{P_3}\}$, which is a permutation of $\Psi$, and differs from $\overline{\Phi}:
A_1'\longrightarrow\{R, \overline{P_1}, \overline{P_2}, \overline{P_3}\}$ on exactly one vertex in $R_l$.
\end{enumerate}
\end{lem}

\begin{thm}\label{final}
If $G$ is an irreducible closed subgroup such that $Aut(\Gamma)\leq G\leq
Sym_{\{l, r\}}(\Gamma)$, then $G=\langle S_{\{l\}}^{H_1}(\Gamma),
S_{\{r\}}^{H_2}(\Gamma)\rangle$ where $H_1, H_2\leq S_3$. If  $G\subset
Sym_{\{l, r\}}(\Gamma)$, then $H_1=H_2$ unless one of the two groups is trivial.
\end{thm}

\begin{proof}
We show that there is a sufficiently complex $(R, \alpha)$ with
$R\subset \Gamma$ and $\alpha\in \Gamma\backslash R$, such that for any $g\in
G$, if $g\upharpoonright R$ is an isomorphism, then
$g\upharpoonright {R\cup\{\alpha\}}$ is a switch w.r.t. $\alpha$ according to some $\sigma\in S_3$. Then by Lemma \ref{Slrgroups} we are done.

Suppose that for any sufficiently complex $(R, w)$, there
exists some $g\in G$ such that $g\upharpoonright R$ is an
isomorphism and $g\upharpoonright R\cup\{w\}$ is not a switch w.r.t. $w$ according to any $\sigma\in S_3$. We
eventually get a contradiction. WLOG, let $w\in R_l$. Choose $N$ such that $i\geq N$ implies $\Gamma_i$ is sufficiently
complex and satisfies the conclusion of Lemma \ref{lemmaA} and
Lemma \ref{lemmaB}. Since $\Gamma_N$ is sufficiently complex, then there exists some $v\in \Gamma \backslash \Gamma_N$ such
that $(\Gamma_N, v)$ is sufficiently complex.  Let $v\in R_l$, and let $R=\Gamma_N\cup \{v\}$. We define the coloring $\chi: [R\backslash \{v\}]^2\longrightarrow \{l, r, P_1, P_2, P_3\}$ as above. The edges between $\{v\}$ and $R\backslash \{v\}$ also induce a vertex
coloring $\Phi_{v}: R\backslash\{v\}\longrightarrow \{L, \overline{P_1}, \overline{P_2}, \overline{P_3}\}$ given for every $a\in R\backslash\{v\}$ by
\begin{itemize}
\item if $a\in R_l$, then $\Phi_v(a)=L$;
\item if $a\in R_r$ and $P_1(v, a)$, then $\Phi_v(a)=\overline{P_1}$;
\item if $a\in R_r$ and $P_2(v, a)$, then $\Phi_v(a)=\overline{P_2}$;
\item if $a\in R_r$ and $P_3(v, a)$, then $\Phi_v(a)=\overline{P_3}$.
\end{itemize}
That is, the color of the vertex $a$ is given by the cross-type of the cross-edge $(a, v)$.

By Lemma \ref{lemmaC}, given $(R\backslash \{v\}, \Phi_v, \chi)$, there are a subgraph $S\subset \Gamma$ with a vertex coloring
$\Phi_{\alpha}: S\longrightarrow\{L, \overline{P_1}, \overline{P_2}, \overline{P_3}\}$ and an edge coloring $\chi: [S]^2\longrightarrow \{l, r, P_1, P_2, P_3\}$ such that for any other vertex coloring $\Psi: S\longrightarrow \{L, \overline{P_1}, \overline{P_2}, \overline{P_3}\}$, if $\Psi$ is not a permutation of $\Psi_{\alpha}$, there exists $R'\subset S$ such that $(R\backslash \{v\}, \Phi_v, \chi)\cong (R', \Phi_{\alpha}, \chi)$, hence $R'\cong R\backslash \{v\}$, and one of the properties $(a)$ and $(b)$ in Lemma \ref{lemmaC} holds. Since $\Gamma$ is random, there must exists a vertex $\alpha\in (\Gamma\backslash S)\cap R_l$ inducing the vertex coloring $\Phi_{\alpha}$.

Note that $\langle R', \Phi_{\alpha}, \chi\rangle\cong \langle R\backslash \{v\}, \Phi_{v}, \chi\rangle$, hence $(R\backslash\{v\})\cup\{v\}\cong R'\cup\{\alpha\}$. Since $(R\backslash\{v\}, v)$ is sufficiently complex, so is $(R', \alpha)$. Then $S\cup \{\alpha\}$ is sufficiently complex since $S\supset R'$. By the assumption at the beginning of the proof, there exists some $g\in G$
which is an isomorphism on $S$, but $g\upharpoonright S\cup\{\alpha\}$ is not a switch w.r.t. $\alpha$ according to any $\sigma\in S_3$. Then the vertex $g(\alpha)$ induces a new coloring
$\Phi_{g(\alpha)}$ on $g(S)$, and $\Phi_{g(\alpha)}$ is not a
permutation of $\Phi_{\alpha}$ since $g\upharpoonright
S\cup\{\alpha\}$ is not a switch w.r.t. $\alpha$ according to any $\sigma\in S_3$. Since $g(S)\cong S$, $g(\alpha)$ also induces a new coloring $\Phi_{g(\alpha)}$ on $S$. Let $\Psi=\Phi_{g(\alpha)}$. Then by Lemma \ref{lemmaC} we can find $R'\subset S$ such that one of the following two
possibilities holds:
\renewcommand{\labelenumi}{(\alph{enumi})}
\begin{enumerate}
\item there is a color $t\in \{\overline{P_1}, \overline{P_2}, \overline{P_3}\}$ such that $\Phi_{g(\alpha)}(x)\neq t$ for
any vertex $x\in R'$.
\item There is a coloring $R'\longrightarrow\{L, \overline{P_1}, \overline{P_2}, \overline{P_3}\}$ which
is a permutation of $\Phi_{g(\alpha)}$, but differs from
$\Phi_{\alpha}$ by exactly one vertex in $R_r$.
\end{enumerate}
Now $R\cong R'\cup\{\alpha\}$ since $\langle R', \Phi_{\alpha}\rangle\cong \langle R\backslash \{v\}, \Phi_{v}\rangle$. If
$(a)$ holds, then there is some $P_i$ ($i=1, 2, 3$) such that there is no edge between $g(\alpha)$ and $g(R')$
with the cross-type $P_i$, contrary to Lemma \ref{lemmaA}. If
$(b)$ holds, then $g\upharpoonright R'\cup\{\alpha\}$ differs from a
switch on exactly one cross-edge, contradicting Lemma
\ref{lemmaB}. Similarly, if we assume $w\in R_r$, we eventually get the contradiction by applying Lemma \ref{lemmaD}. This completes the proof of Theorem \ref{final}.
\end{proof}

%

\section{The Case When $R_l$ and $R_r$ Are Not Preserved}
Next we do not assume that $G$ preserves $R_l$ and $R_r$. Since an element of $G$ either preserves $R_l$ and $R_r$ or switches $R_l$ and
$R_r$, for every reduct $S_X^H{(\Gamma)}$ where $X\subseteq\{l, r\}$, and $H\leq S_3$
(which preserves $R_l, R_r$), we always have a corresponding reduct
${S_X^H}^*{(\Gamma)}$ which preserves the same relation as
$S_X^H{(\Gamma)}$ except it either preserves $R_l$ and $R_r$ or switches
them.

Since the elements in $G$ either preserve $R_l$ and $R_r$ or switch $R_l$ and
$R_r$, we can introduce a relation of on the same side $S(x, y)$ to replace $R_l, R_r$ defined by $S(x, y)\longleftrightarrow (R_l(x)\wedge R_l(y))\vee(R_r(x)\wedge R_r(y))$, and a new binary relation $P_i^*$ to replace $P_i$ defined by $P_i^*(x, y)\longleftrightarrow P_i(x, y) \vee P_i(y, x)$ for $i=1, 2, 3$. So if one forgets $R_l, R_r$ but retains $S(x, y)$ and replaces $P_i$ by $P_i^*$, then either the sides are preserved or switched. Since $\Gamma$ satisfies the extension property $\Theta_n$ for $n\in \mathbb{N}$, we can construct $\rho\in Sym (\Gamma)$ which exchanges the sets $R_r$ and $R_l$ and such that for every $a\in R_l$ and every $b\in R_r$, $P_i(a, b)\longrightarrow P_i(\rho(b), \rho(a))$ where $i=1, 2, 3$. Let $\overline{S_X^H}{(\Gamma)}$ be $\overline{\langle S_X^H{(\Gamma)}, \rho\rangle}$. Note that $[\overline{S_X^H}{(\Gamma)}: S_X^H{(\Gamma)}]=2$, and $S_X^H{(\Gamma)}\triangleleft \overline{S_X^H}{(\Gamma)}$.

\begin{thm}
If $G$ is an irreducible closed subgroup such that $Aut(\Gamma)\leq G\leq
Sym_{\{l, r\}}(\Gamma)$, then $G=\langle S_{\{l\}}^{H_1}(\Gamma),
S_{\{r\}}^{H_2}(\Gamma)\rangle$ or $G=\langle \overline{S_X^{H_1}}(\Gamma),
\overline{S_X^{H_2}}(\Gamma)\rangle$ where $H_1, H_2\leq S_3$. If $G< Sym_{\{l, r\}}(\Gamma)$, then $H_1=H_2$
unless one of the two groups is trivial.
\end{thm}



\bibliography{reducts_graph}
\bibliographystyle{jflnat}

\end{document}